\documentclass{amsart}
\usepackage{amsfonts}
\usepackage{amsmath}
\usepackage{amssymb}
\usepackage{amsthm}

\newtheorem{theorem}{Theorem}[section]
\newtheorem{prop}{Proposition}[section]
\newtheorem{lemma}{Lemma}[section]
\newtheorem{rem}{Remark}[section]
\newtheorem{exmp}{Example}[section]

\begin{document}
\author{Mark Pankov}
\title{Embeddings of Grassmann graphs}
\address{Department of Mathematics and Computer Sciences, University of Warmia and Mazury,
S{\l}oneczna 54, 10-710 Olsztyn, Poland}
\email{pankov@matman.uwm.edu.pl markpankov@gmail.com}

\maketitle

\begin{abstract}
Let $V$ and $V'$ be vector spaces of dimension $n$ and $n'$, respectively.
Let $k\in\{2,\dots,n-2\}$ and $k'\in\{2,\dots,n'-2\}$. We describe all
isometric  and $l$-rigid isometric embeddings of the Grassmann graph $\Gamma_{k}(V)$
in  the Grassmann graph $\Gamma_{k'}(V')$.
\end{abstract}

\section{Introduction}
Let $V$ be an $n$-dimensional vector space.
Denote by ${\mathcal G}_k(V)$ the Grassmannian consisting of $k$-dimensional subspaces of $V$.
Two elements of ${\mathcal G}_k(V)$ are {\it adjacent} if their intersection is $(k-1)$-dimensional.
The {\it Grassmann graph} $\Gamma_{k}(V)$ is the graph whose vertex set is ${\mathcal G}_k(V)$ and whose edges are pairs of adjacent
$k$-dimensional subspaces.
By Chow's theorem \cite{Chow}, if $1<k<n-1$ then every automorphism of $\Gamma_{k}(V)$ is induced by
a semilinear automorphism of $V$ or a semilinear isomorphism of $V$ to the dual vector space $V^{*}$
and the second possibility can be realized only in the case when $n=2k$
(if $k=1,n-1$ then any two distinct vertices of $\Gamma_{k}(V)$ are adjacent and every bijective transformation of ${\mathcal G}_k(V)$
is an automorphism of $\Gamma_{k}(V)$).
Some results closely related with Chow's theorem can be found in \cite{BH,Huang,HuangHavlicek,KMP,Lim,Pankov1}
and we refer \cite{Pankov2} for a survey.

Let $V'$ be an $n'$-dimensional vector space.
We investigate isometric embeddings  of $\Gamma_{k}(V)$ in $\Gamma_{k'}(V')$ under assumption that $1<k<n-1$ and $1<k'<n'-1$
(the case $k=k'$ was considered in \cite{KMP}).
Then $n,n'\ge 4$ and the existence of such embeddings implies that
the diameter of $\Gamma_{k}(V)$ is not greater than the  diameter of $\Gamma_{k'}(V')$, i.e.
$$\min\{k,n-k\}\le \min\{k',n'-k'\}.$$
In the case when $k\le n-k$,
we show that every isometric embedding of $\Gamma_{k}(V)$ in $\Gamma_{k'}(V')$
is induced by a semilinear $(2k)$-embedding (a semilinear injection such that the image of every independent $(2k)$-element  subset is independent)
of $V$ in $V'/S$, where $S$ is a $(k'-k)$-dimensional subspace of $V'$,
or a semilinear $(2k)$-embedding of $V$ in $U^{*}$, where $U$ is a $(k'+k)$-dimensional subspace of $V'$
(Theorem \ref{theorem4-1}).
If $k>n-k$ then there are isometric embeddings of $\Gamma_{k}(V)$ in $\Gamma_{k'}(V')$
which can not be induced by semilinear mappings of $V$ to some vector spaces.

Our second result (Theorem \ref{theorem6-1}) is related with so-called $l$-rigid embeddings.
An embedding $f$ of a graph $\Gamma$ in a graph $\Gamma'$ is {\it rigid } if
for every automorphism $g$ of $\Gamma$ there is an automorphism $g'$ of $\Gamma'$
such that the diagram
$$
\begin{array}{ccccccccccc}
\Gamma & \stackrel{f}{\longrightarrow} & \Gamma'\\
\downarrow\lefteqn{g}&&\downarrow\lefteqn{g'}\\
\Gamma & \stackrel{f}{\longrightarrow} & \Gamma'
\end{array}
$$
is commutative, roughly speaking, every automorphism of $\Gamma$ can be extended to an automorphism of $\Gamma'$.
We say that an embedding of $\Gamma_{k}(V)$ in $\Gamma_{k'}(V')$ is $l$-{\it rigid} if every automorphism of
$\Gamma_{k}(V)$ induced by a linear automorphism of $V$ can be extended to the automorphism of $\Gamma_{k'}(V')$
induced by a  linear automorphism of $V'$.

In the case when $n=2k$, every isometric embedding of $\Gamma_{k}(V)$ in $\Gamma_{k'}(V')$ is $l$-rigid.
In general case (we do not require that $k\le n-k$),
every $l$-rigid isometric embedding of $\Gamma_{k}(V)$ in $\Gamma_{k'}(V')$
is induced by a semilinear embedding (a semilinear injection transferring independent subsets to independent subsets)
of $V$ in $V'/S$, where $S$ is a $(k'-k)$-dimensional subspace of $V'$,
or a semilinear embedding of $V$ in $U^{*}$, where $U$ is a $(k'+k)$-dimensional subspace of $V'$.
The proof of this result is based on a characterization of semilinear embeddings (Theorem \ref{theorem5-1}).

Using \cite{Kreuzer}, we establish the existence of isometric embeddings of $\Gamma_{k}(V)$ in $\Gamma_{k'}(V')$
which are not $l$-rigid (Example \ref{exmp6-4}).

\section{Basic facts and definitions}

\subsection{}
Let $\Gamma$ be a connected graph. A subset in the vertex set of $\Gamma$
formed by mutually adjacent vertices is called  a {\it clique}.
Using Zorn lemma, we can show that every clique is contained in a maximal clique.
The {\it distance} $d(v,w)$ between two vertices $v$ and $w$ of $\Gamma$ is defined as the smallest number $i$
such that there exists a path of length $i$ (a path consisting of $i$ edges) connecting $v$ and $w$.
The {\it diameter} of $\Gamma$ is the maximum of all distances $d(v,w)$.

An injective mapping of the vertex set of $\Gamma$ to the vertex set of a graph $\Gamma'$
is called an {\it embedding} of $\Gamma$ in $\Gamma'$ if
vertices of $\Gamma$ are adjacent only in the case when
their images are adjacent vertices of $\Gamma'$.
Every surjective embedding is an isomorphism.
An embedding is said to be {\it isometric} if it preserves the distance between any two vertices.

\subsection{}
Let $k\in\{1,\dots,n-1\}$.
Consider incident subspaces $S,U\subset V$ such that
$$\dim S<k<\dim U$$
and denote by $[S,U]_{k}$ the set formed by all $X\in {\mathcal G}_{k}(V)$ satisfying $S\subset X \subset U$.
In the case when $S=0$ or $U=V$, this set will be denoted by $\langle U]_{k}$ or $[S\rangle_{k}$, respectively.
The {\it Grassmann space} ${\mathfrak G}_{k}(V)$ is the partial linear space whose point set is ${\mathcal G}_{k}(V)$
and whose lines are subsets of type
$$[S,U]_{k},\;S\in {\mathcal G}_{k-1}(V),\;U\in {\mathcal G}_{k+1}(V).$$
It is clear that ${\mathfrak G}_{1}(V)=\Pi_{V}$ and ${\mathfrak G}_{n-1}(V)=\Pi^{*}_{V}$
(we denote by $\Pi_{V}$ the projective space associated with $V$ and write $\Pi^{*}_{V}$ for the corresponding dual projective space).
Two distinct points of ${\mathfrak G}_{k}(V)$ are collinear (joined by a line) if and only if they are adjacent vertices
of the Grassmann graph $\Gamma_{k}(V)$.

If $1<k<n-1$ then there are precisely the following two types of maximal cliques of $\Gamma_{k}(V)$:
\begin{enumerate}
\item[(1)] the {\it top} $\langle U]_{k}$, $U\in {\mathcal G}_{k+1}(V)$,
\item[(2)] the {\it star} $[S\rangle_{k}$, $S\in {\mathcal G}_{k-1}(V)$.
\end{enumerate}
The top  $\langle U]_{k}$ and the star $[S\rangle_{k}$ together with the lines contained in them are projective spaces.
The first projective space is $\Pi^{*}_{U}$ and the second can be identified with $\Pi_{V/S}$.

The distance $d(X,Y)$ between $X,Y\in {\mathcal G}_k(V)$ in the graph $\Gamma_{k}(V)$ is equal to
$$k-\dim(X\cap Y)$$
and the diameter of $\Gamma_{k}(V)$ is equal to $\min\{k,n-k\}$.
The annihilator mapping (which transfers every subspace $S\subset V$ to the annihilator $S^{0}\subset V^{*}$)
induces an isomorphism between $\Gamma_{k}(V)$ and  $\Gamma_{n-k}(V^{*})$.

\subsection{}
Throughout the paper we will suppose that $V$ and $V'$ are left vector spaces over division rings $R$ and $R'$, respectively.
An additive mapping $l:V\to V'$
is said to be {\it semilinear} if there exists a homomorphism
$\sigma:R\to R'$ such that
$$l(ax)=\sigma(a)l(x)$$
for all $x\in V$ and all $a\in R$.
If $l$ is non-zero then  there is only one homomorphism satisfying this condition.
Every non-zero homomorphism of $R$ to $R'$ is injective.

Every semilinear injection of $V$ to $V'$ induces a mapping of ${\mathcal G}_{1}(V)$
to ${\mathcal G}_{1}(V')$ which transfers lines of $\Pi_{V}$ to subsets in lines of $\Pi_{V'}$
(note that this mapping is not necessarily injective). We will use the following version of the
Fundamental Theorem of Projective Geometry \cite{FaureFrolicher,Faure,Havlicek}, see also \cite[Theorem 1.4]{Pankov2}.

\begin{theorem}[C. A. Faure, A. Fr\"{o}licher, H. Havlicek]\label{theorem2-1}
Every mapping of ${\mathcal G}_{1}(V)$ to ${\mathcal G}_{1}(V')$
transferring lines of $\Pi_{V}$ to subsets in lines of $\Pi_{V'}$
is induced by a semilinear injection of $V$ to $V'$.
\end{theorem}

A semilinear mapping of $V$ to $V'$ is called a {\it semilinear isomorphism} 
if it is bijective and the associated homomorphism of $R$ to $R'$ is an isomorphism.
If $u$ is a semilinear automorphism of $V$ then the mapping $u_{k}$
sending every $X\in {\mathcal G}_{k}(V)$ to $u(X)$ is an automorphism of $\Gamma_{k}(V)$.
If $n=2k$ and $v:V\to V^{*}$ is a semilinear isomorphism  then
the bijection transferring every $X\in {\mathcal G}_{k}(V)$ to the annihilator of $v(X)$
is an automorphism of $\Gamma_{k}(V)$.

\begin{theorem}[W. L. Chows \cite{Chow}]
Every automorphism of $\Gamma_{k}(V)$, $1<k<n-1$ is induced by a semilinear automorphism of $V$ or
a semilinear isomorphism of $V$ to $V^{*}$; the second possibility can be realized only in the case when $n=2k$.
\end{theorem}

\section{General properties of embeddings}
Let $k\in \{2,\dots,n-2\}$ and $k'\in \{2,\dots,n'-2\}$.
Every embedding of $\Gamma_{k}(V)$ in $\Gamma_{k'}(V')$
transfers maximal cliques of $\Gamma_{k}(V)$ to subsets in maximal cliques of $\Gamma_{k'}(V')$;
moreover, every maximal clique of $\Gamma_{k'}(V')$ contains at most one image of a maximal clique of $\Gamma_{k}(V)$
(otherwise, the preimages of some adjacent vertices of $\Gamma_{k'}(V')$ are non-adjacent which is impossible).

\begin{prop}\label{prop3-1}
For any embedding of $\Gamma_{k}(V)$ in $\Gamma_{k'}(V')$
the image of every maximal clique of $\Gamma_{k}(V)$ is contained in precisely one maximal clique of $\Gamma_{k'}(V')$.
\end{prop}

\begin{proof}
Let $f$ be an embedding of $\Gamma_{k}(V)$ in $\Gamma_{k'}(V')$.
Suppose that ${\mathcal X}$ is a maximal clique of $\Gamma_{k}(V)$ and $f({\mathcal X})$ is contained
in two distinct  maximal cliques of $\Gamma_{k'}(V')$. Since the intersection of two distinct maximal cliques
is empty or a one-element set or a line, there exist $S\in{\mathcal G}_{k'-1}(V')$ and $U\in{\mathcal G}_{k'+1}(V')$
such that
\begin{equation}\label{eq3-1}
f({\mathcal X})\subset[S,U]_{k'}.
\end{equation}
We take any maximal clique ${\mathcal Y}\ne{\mathcal X}$ of $\Gamma_{k}(V)$ which intersects ${\mathcal X}$ in a line
and consider a maximal clique ${\mathcal Y}'$ of $\Gamma_{k'}(V')$ containing $f({\mathcal Y})$.
The inclusion \eqref{eq3-1} guarantees that the line
$[S,U]_{k'}$ intersects $f({\mathcal Y})\subset{\mathcal Y}'$ in a set containing more than one element.
Then $[S,U]_{k'}\subset{\mathcal Y}'$
(a line is contained in a maximal clique or intersects it in at most one element).
So, the maximal clique ${\mathcal Y}'$ contains the images of both ${\mathcal X}$ and ${\mathcal Y}$
which are distinct maximal cliques of $\Gamma_{k}(V)$, a contradiction.
\end{proof}

It was noted above that the intersection of two distinct maximal cliques of $\Gamma_{k}(V)$ is empty or a one-element set or a line.
The latter possibility can be realized only in the case when the maximal cliques are of different types ---
one of them is a star and the other is a top. For any distinct maximal cliques ${\mathcal X},{\mathcal Y}$
of $\Gamma_{k}(V)$ there is a sequence of maximal cliques of $\Gamma_{k}(V)$
$${\mathcal X}={\mathcal X}_{0},{\mathcal X}_{1},\dots,{\mathcal X}_{i}={\mathcal Y}$$
such that
${\mathcal X}_{j-1}\cap{\mathcal X}_{j}$ is a line for every $j\in\{1,\dots,i\}$.
This implies the following.

\begin{prop}\label{prop3-2}
For every embedding of $\Gamma_{k}(V)$ in $\Gamma_{k'}(V')$ one of the following possibilities is realized:
\begin{enumerate}
\item[{\rm (A)}] stars go to subsets of stars and tops go to subsets of tops,
\item[{\rm (B)}] stars go to subsets of tops and tops go to subsets of stars.
\end{enumerate}
\end{prop}

We say that an embedding is of {\it type} (A) or (B)
if the corresponding possibility is realized.

If an embedding $f$ of $\Gamma_{k}(V)$ in $\Gamma_{k'}(V')$ is of type $(A)$
then the embedding of $\Gamma_{k}(V)$ in $\Gamma_{n'-k'}(V'^{*})$ sending every $X\in {\mathcal G}_{k}(V)$
to the annihilator of $f(X)$ is of type (B).

\section{Isometric embeddings}
\setcounter{equation}{0}

A semilinear injection of $V$ to $V'$ is said to be a {\it semilinear $m$-embedding} if
the image of every independent $m$-element subset is independent.
The existence of such mappings implies that $m\le \min\{n,n'\}$.
In the case when $n\le n'$,  semilinear $n$-embeddings of $V$ in $V'$ will be called {\it semilinear embeddings}.

\begin{rem}\label{rem-main1}{\rm
By \cite{Kreuzer}, there exist fields $F$ and $F'$ such that for any natural numbers $p$ and $q$
there is a semilinear $p$-embedding of $F^{p+q}$ in $F'^{p}$.
It is clear that such semilinear $p$-embeddings can not be $(p+1)$-embeddings.
}\end{rem}

Let $l:V\to V'$ be a semilinear $m$-embedding.
For every $p\in\{1,\dots,m\}$ and every $p$-dimensional subspace $X\subset V$
the dimension of the subspace spanned by $l(X)$ is equal to $p$. So, we have the mapping
$$l_{p}:{\mathcal G}_{p}(V)\to {\mathcal G}_{p}(V')$$
$$X\to \langle l(X)\rangle.$$
By \cite[Proposition 2]{KMP}, if $2p\le \min\{n,n'\}$ and $l:V\to V'$ is a $(2p)$-embedding then
$l_p$ is an isometric embedding of $\Gamma_{p}(V)$ in  $\Gamma_{p}(V')$.

Let $k\in \{2,\dots,n-2\}$ and $k'\in \{2,\dots,n'-2\}$.
The existence of isometric embeddings of $\Gamma_{k}(V)$ in  $\Gamma_{k'}(V')$ implies that
$$\min\{k,n-k\}\le \min\{k',n'-k'\}$$
(the diameter of $\Gamma_{k}(V)$ is not greater than the  diameter of $\Gamma_{k'}(V')$).
In the next three examples we suppose that $k\le n-k$, i.e.
\begin{equation}\label{eq4-0}
k\le \min\{k',n-k,n'-k'\}.
\end{equation}

\begin{exmp}\label{exmp4-1}{\rm
Let $S\in {\mathcal G}_{k'-k}(V')$. By \eqref{eq4-0},
$$\dim (V'/S)=n'-k'+k\ge 2k.$$
If $l:V\to V'/S$ is a semilinear $(2k)$-embedding then $l_{k}$ is an isometric embedding of
$\Gamma_{k}(V)$ in  $\Gamma_{k}(V'/S)$.
Let $\pi$ be the natural isometric embedding of $\Gamma_{k}(V'/S)$ in $\Gamma_{k'}(V')$
(which transfers every $k$-dimensional subspace of $V'/S$ to the corresponding $k'$-dimensional subspace of $V'$).
Then $\pi l_{k}$ is an isometric embedding of $\Gamma_{k}(V)$ in $\Gamma_{k'}(V')$ of type (A).
}\end{exmp}

\begin{exmp}\label{exmp4-2}{\rm
Let $U\in {\mathcal G}_{k'+k}(V')$ (by \eqref{eq4-0}, we have $k'+k\le n'$).
If $v:V\to U^{*}$ is a  semilinear $(2k)$-embedding
then $v_{k}$ is an isometric embedding of $\Gamma_{k}(V)$ in $\Gamma_{k}(U^{*})$.
By duality, it can be considered as an isometric embedding of $\Gamma_{k}(V)$ in $\Gamma_{k'}(U)$.
Since $U$ is contained in $V'$, we get an isometric embedding of $\Gamma_{k}(V)$ in $\Gamma_{k'}(V')$ of type (B).
}\end{exmp}

\begin{exmp}\label{exmp4-3}{\rm
Suppose that $n=2k$ and $S\in {\mathcal G}_{k'-k}(V')$, $U\in {\mathcal G}_{k'+k}(V')$ are incident.
Then
$$\dim (U/S)=2k=n.$$
By Example \ref{exmp4-1}, every semilinear embedding of $V$ in $U/S\subset V'/S$ induces an isometric embedding of $\Gamma_{k}(V)$ in $\Gamma_{k'}(V')$.
If $w:V\to (U/S)^{*}$ is a  semilinear embedding then
$w_{k}$ is an isometric embedding of $\Gamma_{k}(V)$ in $\Gamma_{k}((U/S)^{*})$ and, by duality,
it can be considered as an isometric embedding of $\Gamma_{k}(V)$ in $\Gamma_{k}(U/S)$.
As in Example \ref{exmp4-1}, we construct an isometric embedding of $\Gamma_{k}(V)$ in $\Gamma_{k'}(V')$.
}\end{exmp}

\begin{theorem}\label{theorem4-1}
Let $f$ be an isometric embedding of $\Gamma_{k}(V)$ in $\Gamma_{k'}(V')$.
If $k\le n-k$ then one of the following possibilities is realized:
\begin{enumerate}
\item[{\rm (A)}] there is $S\in {\mathcal G}_{k'-k}(V')$
such that $f$ is induced by a semilinear $(2k)$-embedding of $V$ in $V'/S$, see Example \ref{exmp4-1};
\item[{\rm (B)}] there is $U\in {\mathcal G}_{k'+k}(V')$
such that $f$ is induced by a semilinear  $(2k)$-embedding of $V$ in $U^{*}$, see Example \ref{exmp4-2}.
\end{enumerate}
In the case when $n=2k$, there are incident $S\in {\mathcal G}_{k'-k}(V')$ and $U\in {\mathcal G}_{k'+k}(V')$
such that $f$ is induced by a semilinear embedding of $V$ in $U/S$ or a semilinear embedding of $V$ in $(U/S)^{*}$,
see Example \ref{exmp4-3}.
\end{theorem}

\begin{rem}{\rm
Suppose that $f$ is an isometric embedding of $\Gamma_{k}(V)$ in $\Gamma_{k'}(V')$ and $k>n-k$.
Consider the mapping which sends every $X\in {\mathcal G}_{n-k}(V^{*})$ to $f(X^{0})$ (recall that $X^{0}$ is the annihilator of $X$).
This is an isometric embedding of $\Gamma_{n-k}(V^{*})$ in $\Gamma_{k'}(V')$.
Since $n-k<n-(n-k)$, it is induced by a semilinear $2(n-k)$-embedding of $V^{*}$ in one of vector spaces described above.
}\end{rem}

\begin{proof}[Proof of Theorem \ref{theorem4-1}]
Suppose that $f$ is an embedding of type (A). By Subsection 3, there exists an injective mapping
$$f_{k-1}:{\mathcal G}_{k-1}(V)\to {\mathcal G}_{k'-1}(V')$$
such that
$$f([X\rangle_k)\subset [f_{k-1}(X)\rangle_{k'}\;\;\;\;\;\forall\;X\in {\mathcal G}_{k-1}(V).$$
Then
$$
f_{k-1}(\langle Y]_{k-1})\subset \langle f(Y)]_{k'-1}\;\;\;\;\;\forall\;Y\in {\mathcal G}_{k}(V).
$$
Since for any two adjacent vertices there is a top containing them,
the latter inclusion implies that $f_{k-1}$ is adjacency preserving.
Thus for any $X,Y\in {\mathcal G}_{k-1}(V)$ we have
$$d(X,Y)\ge d(f_{k-1}(X),f_{k-1}(Y)).$$
We prove the inverse inequality.

The condition $2k\le n$ implies the existence of $X',Y'\in {\mathcal G}_{k}(V)$
such that $X\subset X'$, $Y\subset Y'$ and
$$X\cap Y=X'\cap Y'.$$
Then
\begin{equation}\label{eq4-1}
d(X,Y)=d(X',Y')-1
\end{equation}
(indeed, $d(X,Y)=k-1-\dim(X\cap Y)=k-1-\dim(X'\cap Y')=d(X',Y')-1$).
Since $f_{k-1}$ is induced by $f$,
we have
$$f_{k-1}(X)\subset f(X')\;\mbox{ and }\;f_{k-1}(Y)\subset f(Y')$$
which guarantees that
\begin{equation}\label{eq4-2}
\dim(f_{k-1}(X)\cap f_{k-1}(Y))\le \dim(f(X')\cap f(Y')).
\end{equation}
Using \eqref{eq4-1} and \eqref{eq4-2}, we get the following
$$d(X,Y)=d(X',Y')-1=d(f(X'),f(Y'))-1=k'-1-\dim(f(X')\cap f(Y'))\le$$
$$k'-1-\dim(f_{k-1}(X)\cap f_{k-1}(Y))=d(f_{k-1}(X),f_{k-1}(Y)).$$

So, $f_{k-1}$ is an isometric embedding of $\Gamma_{k-1}(V)$ in $\Gamma_{k'-1}(V')$.
This is an embedding of type (A) (it was established above that $f_{k-1}$ sends tops to subsets of tops).
Step by step, we construct a sequence of isometric embeddings
$$f_{i}:{\mathcal G}_{i}(V)\to {\mathcal G}_{k'-k+i}(V),\;\;\;\;i=k,\dots,1,$$
of $\Gamma_{i}(V)$ in $\Gamma_{k-k'+i}(V')$ such that $f_{k}=f$ and we have
$$
f_{i}([X\rangle_{i})\subset[ f_{i-1}(X)\rangle_{k'-k+i}
\;\;\;\;\;\forall\; X\in {\mathcal G}_{i-1}(V)
$$
and
\begin{equation}\label{eq4-3}
f_{i-1}(\langle Y]_{i-1})\subset\langle f_i(Y)]_{k'-k+i-1}
\;\;\;\;\;\forall\;Y\in {\mathcal G}_{i}(V)
\end{equation}
if $i>1$.

The image of $f_1$ is a clique of $\Gamma_{k'-k+1}(V')$.
This clique can not be contained in any top (otherwise, there is $X'\in {\mathcal G}_{k'-k+2}(V')$ such that
$f_{2}(X)=X'$ for every $X\in {\mathcal G}_{2}(V)$ and $f_{2}$ is not injective).
Therefore, there is $S\in {\mathcal G}_{k'-k}(V')$ such that the image of $f_1$ is contained in the star $[S\rangle_{k'-k+1}$.

By \eqref{eq4-3}, $f_1$ transfers lines of $\Pi_{V}$ to subsets of lines contained in $[S\rangle_{k'-k+1}$.
It was noted in Subsection 2.2 that  the star $[S\rangle_{k'-k+1}$ (together with all lines contained in it) can be identified
with the projective space $\Pi_{V'/S}$.
Theorem \ref{theorem2-1} shows that $f_{1}$ is induced by a semilinear injection $l:V\to V'/S$.

Using \eqref{eq4-3}, we establish that
$$f_{1}(\langle X]_{1})\subset \langle f(X)]_{k'-k+1}
\;\;\;\;\;\forall\;X\in {\mathcal G}_{k}(V).$$
On the other hand,
$$f_{1}(\langle X]_{1})\subset \langle \pi(\langle l(X)\rangle)]_{k'-k+1}
\;\;\;\;\;\forall\;X\in {\mathcal G}_{k}(V),$$
where $\pi$ is the mapping which transfers  every subspace of $V'/S$ to the corresponding subspace of $V'$.
Since the intersection of two distinct tops contains at most one element, we get
$$f(X)=\pi(\langle l(X)\rangle)
\;\;\;\;\;\forall\;X\in {\mathcal G}_{k}(V)$$
which means that $f=\pi l_k$.

Every independent $(2k)$-element subset $A\subset V$ can be presented as the disjoint union of two
independent $k$-element subsets $A_1$ and $A_2$.
Then
$$d(f(\langle A_1\rangle), f(\langle A_2\rangle))=d(\langle A_1\rangle, \langle A_2\rangle)=k$$
which means that $\langle l(A_1)\rangle$ and $\langle l(A_2)\rangle$
are $k$-dimensional subspaces of $V'/S$ intersecting in $0$. Hence the subspace spanned by $l(A_{1}\cup A_{2})=l(A)$
is $(2k)$-dimensional.

So, $l$ is a semilinear $(2k)$-embedding of $V$ in $V'/S$  and $f$ is as in Example \ref{exmp4-1}.
In the case when $n=2k$, $l$ is a semilinear embedding and
the image of $l$ is contained in $U/S$, where $U\in {\mathcal G}_{k'+k}(V')$.

Now suppose that $f$ is an embedding of type (B).
By duality, $f$ can be considered as an isometric embedding of $\Gamma_{k}(V)$ in $\Gamma_{n'-k'}(V'^{*})$
(this embedding sends every $X\in{\mathcal G}_{k}(V)$ to the annihilator of $f(X)$).
We get an embedding  of type (A) and its image is contained in
$$[S'\rangle_{n'-k'},\;\;S'\in {\mathcal G}_{n'-k'-k}(V'^{*});$$
in the case when $n=2k$, the image is contained in
$$[S',U']_{n'-k'},\;\;S'\in {\mathcal G}_{n'-k'-k}(V'^{*}),\;U'\in {\mathcal G}_{n'-k'+k}(V'^{*}).$$
The image of $f$ is contained in $\langle U]_{k'}$, where $U\in{\mathcal G}_{k'+k}(V')$ is the annihilator of $S'$.
Thus $f$ is an isometric embedding of $\Gamma_{k}(V)$ in $\Gamma_{k'}(U)$. By duality, $f$ can be considered as
an isometric embedding of $\Gamma_{k}(V)$ in $\Gamma_{k}(U^*)$ of type (A). Hence it is induced by
a semilinear $(2k)$-embedding of $V$ in $U^{*}$, i.e. $f$ is as in Example \ref{exmp4-2}.

If $n=2k$ then  the image of $f$ is contained in $[S,U]_{k'}$, where
$S\in{\mathcal G}_{k'-k}(V')$ and $U\in{\mathcal G}_{k'+k}(V')$ are the annihilators of $U'$ and $S'$, respectively.
This means that $f=\pi f'$, where $f'$ is an isometric embedding of $\Gamma_{k}(V)$ in $\Gamma_{k}(U/S)$ of type (B)
and $\pi$ is the mapping which transfers  every subspace of $V'/S$ to the corresponding subspace of $V'$.
Since
$$\dim(U/S)=2k,$$
$f'$ can be considered as an isometric embedding of $\Gamma_{k}(V)$ in $\Gamma_{k}((U/S)^*)$ of type (A).
The latter embedding is induced by a semilinear embedding of $V$ in $(U/S)^*$ and $f$ is as in Example \ref{exmp4-3}.
\end{proof}

\begin{rem}{\rm
Using the same idea, the author describes
the images of isometric embeddings of Johnson graphs in Grassmann graphs \cite[Theorem 4]{Pankov3}.
}\end{rem}
\section{Characterization of semilinear embeddings}
\setcounter{equation}{0}

\begin{theorem}\label{theorem5-1}
Let $l:V\to V'$ be a semilinear injection.
Then $l$ is a semilinear embedding if and only if 
for every linear automorphism  $u\in {\rm GL}(V)$ there is a linear automorphism  $u'\in {\rm GL}(V')$
such that the diagram 
\begin{equation}\label{eq5-1}
\begin{array}{ccccccccccc}
V & \stackrel{l}{\longrightarrow} & V'\\
\downarrow\lefteqn{u}&&\downarrow\lefteqn{u'}\\
V & \stackrel{l}{\longrightarrow} & V'
\end{array}
\end{equation}
is commutative.
\end{theorem}

To prove Theorem \ref{theorem5-1} we use the following result.

\begin{theorem}[M. Pankov \cite{Pankov3}]\label{theorem5-2}
For a finite subset ${\mathcal X}\subset {\mathcal G}_{1}(V)$ the following conditions are equivalent:
\begin{enumerate}
\item[$\bullet$] every permutation on ${\mathcal X}$ is induced by a semilinear automorphism of $V$,
\item[$\bullet$] ${\mathcal X}$ is a simplex or an independent subset of $\Pi_{V}$.
\end{enumerate}
\end{theorem}

Recall that $Q_{1},\dots,Q_{m}\in {\mathcal G}_{1}(V)$ form an independent subset of $\Pi_{V}$
if non-zero vectors $x_{1}\in Q_{1},\dots,x_{m}\in Q_{m}$
form an independent subset of $V$.
An $(m+1)$-element subset ${\mathcal X}\subset {\mathcal G}_{1}(V)$ is called an $m$-{\it simplex} of $\Pi_{V}$ if
it is not independent and every $m$-element subset of ${\mathcal X}$ is independent \cite[Section III.3]{Baer}.

\begin{rem}\label{rem5-1}{\rm
If $x_{1},\dots,x_{m}\in V$ and $\langle x_{1}\rangle,\dots,\langle x_{m}\rangle$ form an $(m-1)$-simplex
then $x_{1},\dots,x_{m-1}$ are linearly independent and
$x_{m}=\sum^{m-1}_{i=1}a_{i}x_{i},$
where every scalar $a_i$ is non-zero.
}\end{rem}

\begin{proof}[Proof of Theorem \ref{theorem5-1}]
Suppose that $l:V\to V'$ is a semilinear embedding.
Let $\{x_i\}^{n}_{i=1}$ be a base of $V$. For every vector $x=\sum^{n}_{i=1} a_ix_i$ and every linear automorphism $u\in {\rm GL}(V)$ we have
$$l(x)=\sum^{n}_{i=1}\sigma(a_{i})l(x_i)
\;\mbox{ and }\;
l(u(x))=\sum^{n}_{i=1}\sigma(a_{i})l(u(x_i)),$$
where $\sigma:R\to R'$ is the homomorphism associated with $l$.
Since $\{l(x_i)\}^{n}_{i=1}$ and $\{l(u(x_i))\}^{n}_{i=1}$ both are independent subsets of $V'$, the diagram
\eqref{eq5-1} is commutative for any linear automorphism $u'\in {\rm GL}(V')$ transferring every $l(x_i)$ to $l(u(x_i))$.

Conversely, suppose that 
for every linear automorphism  $u\in {\rm GL}(V)$ there is a linear automorphism  $u'\in {\rm GL}(V')$
such that the diagram \eqref{eq5-1} is commutative.
Let $B=\{x_i\}^{n}_{i=1}$ be a base of $V$.
Every permutation on the associated base of $\Pi_{V}$ is induced by a linear automorphism of $V$.
Then, by our assumption, every permutation on
the set
$${\mathcal X}(B):=\{\langle l(x_{i})\rangle\}^{n}_{i=1}$$
is induced by a linear automorphism of $V'$.
Theorem \ref{theorem5-2} implies that ${\mathcal X}(B)$ is an $(n-1)$-simplex or an independent subset of $\Pi_{V'}$.
In the second case, $l(x_1),\dots,l(x_{n})$ are linearly independent and $l$ is a semilinear embedding.

If ${\mathcal X}(B)$ is an $(n-1)$-simplex then ${\mathcal X}(B')$ is an $(n-1)$-simplex for every base $B'\subset V$ 
(indeed, 
if there is a base $B'\subset V$ such that ${\mathcal X}(B')$ is an independent subset of $\Pi_{V'}$
then $l$ is a semilinear embedding and ${\mathcal X}(B')$ is an independent subset for every base $B'\subset V$).
We need to show that this possibility can not be realized.

So, suppose that ${\mathcal X}(B)$ is an $(n-1)$-simplex.
By Remark \ref{rem5-1}, $\{l(x_i)\}^{n-1}_{i=1}$ is an independent subset of $V'$ and
\begin{equation}\label{eq5-2}
l(x_{n})=\sum^{n-1}_{i=1}a_il(x_i),
\end{equation}
where every scalar $a_i$ is non-zero.
Let $u$ be an automorphism of $V$. Then $u(B)$ is a base of $V$ and ${\mathcal X}(u(B))$ is an $(n-1)$-simplex.
The latter implies that
$\{l(u(x_i))\}^{n-1}_{i=1}$ is an independent subset of $V'$ and
\begin{equation}\label{eq5-3}
l(u(x_{n}))=\sum^{n-1}_{i=1}b_il(u(x_i)),
\end{equation}
where every scalar $b_i$ is non-zero.
If $u'$ is a linear automorphism of $V'$ such that the diagram \eqref{eq5-1} is commutative then
it transfers every $l(x_i)$ to $l(u(x_i))$ and \eqref{eq5-2}, \eqref{eq5-3} show that
$a_{i}=b_{i}$ for every $i$, i.e.
$$
l(u(x_{n}))=\sum^{n-1}_{i=1}a_il(u(x_i)).
$$
Consider a linear automorphism $v\in {\rm GL}(V)$ such that $u(x_{i})=v(x_{i})$ if $i\le n-1$ and 
$u(x_{n})\ne v(x_{n})$.
By the arguments given above,
$$l(v(x_{n}))=\sum^{n-1}_{i=1}a_il(v(x_i))=\sum^{n-1}_{i=1}a_il(u(x_i))=l(u(x_{n}))$$
which contradicts the injectivity of $l$.
\end{proof}

\section{$l$-rigid isometric embeddings}
\setcounter{equation}{0}

As above, we suppose that $k\in \{2,\dots,n-2\}$ and $k'\in \{2,\dots,n'-2\}$.
Let $f$ be an embedding of $\Gamma_{k}(V)$ in $\Gamma_{k'}(V')$.
We say that $f$ is $l$-{\it rigid} if 
for every linear automorphism $u\in {\rm GL}(V)$ there is a linear automorphism $u'\in {\rm GL}(V')$
such that the diagram 
$$
\begin{array}{ccccccccccc}
\Gamma_{k}(V) & \stackrel{f}{\longrightarrow} & \Gamma_{k'}(V')\\
\downarrow\lefteqn{u_k}&&\downarrow\lefteqn{u'_{k'}}\\
\Gamma_{k}(V) & \stackrel{f}{\longrightarrow} & \Gamma_{k'}(V')
\end{array}
$$
is commutative, i.e.
every automorphism of $\Gamma_{k}(V)$  induced by a linear automorphism of $V$
can be extended to the automorphism of $\Gamma_{k'}(V')$ induced by a linear automorphism of $V'$.

\begin{exmp}\label{exmp6-1}{\rm
Suppose that $n=2k$ and $V$ is a subspace of $V'$.
We also require that the associated division ring is isomorphic to the opposite division ring.
This implies the existence of semilinear isomorphisms of $V$ to $V^{*}$.
Then the natural embedding of $\Gamma_{k}(V)$ in $\Gamma_{k}(V')$ is not rigid,
since every automorphism of $\Gamma_{k}(V)$ induced by a semilinear isomorphism of $V$ to $V^{*}$
can not be extended to an automorphism of $\Gamma_{k}(V')$. 
It is clear that this embedding is $l$-rigid.
}\end{exmp}

Let $u$ be a linear automorphism of $V$ and let $u^*$ be the adjoint linear automorphism of $V^*$.
The linear automorphism ${\check u}:=(u^{*})^{-1}$ is called the {\it contragradient} of $u$.
It transfers the annihilator of every subspace $S\subset V$ to
the annihilator of $u(S)$ \cite[Section 1.3.3]{Pankov2}.

\begin{lemma}\label{lemma6-1}
If $f$ is an $l$-rigid embedding of $\Gamma_{k}(V)$ in $\Gamma_{k'}(V')$
then the same holds for the following two embeddings:
\begin{enumerate}
\item[$\bullet$]
the embedding of $\Gamma_{k}(V)$ in $\Gamma_{n'-k'}(V'^{*})$
transferring every $X\in {\mathcal G}_k(V)$ to the annihilator of $f(X)$,
\item[$\bullet$]
the embedding of $\Gamma_{n-k}(V^{*})$ in $\Gamma_{k'}(V')$
transferring every $X\in {\mathcal G}_{n-k}(V^{*})$ to $f(X^{0})$.
\end{enumerate}
\end{lemma}

\begin{proof}
Let $u$ be a linear automorphism of $V$.
Then $u_k$ is an automorphism of $\Gamma_{k}(V)$.
Suppose that it is extendable to the automorphism of $\Gamma_{k'}(V')$ induced by a linear automorphism $v\in {\rm GL}(V')$.
Then ${\check v}$ defines the extension of $u_k$ to an automorphism of $\Gamma_{n'-k'}(V'^{*})$.

Consider the second embedding.
The  contragradient ${\check u}$ is a linear automorphism of $V^{*}$ and ${\check u}_{n-k}$ is an automorphism
of $\Gamma_{n-k}(V^{*})$. It is clear that ${\check u}_{n-k}$ is extendable to the automorphism of $\Gamma_{k'}(V')$ induced by $v$.
This completes our proof,
since every linear automorphism of $V^{*}$ is the contragradient of a linear automorphism of $V$.
\end{proof}

We give two examples of $l$-rigid isometric embeddings of $\Gamma_{k}(V)$ in $\Gamma_{k'}(V')$.
In contrast with the previous section, we do not require that $k\le n-k$.

\begin{exmp}\label{exmp6-2}{\rm
Suppose that
$$k\le k'\;\mbox{ and }\;n-k\le n'-k'.$$
If $S\in {\mathcal G}_{k'-k}(V')$ then
$$\dim V'/S=n'-k'+k\ge n.$$
For every semilinear embedding $l:V\to V'/S$ the mapping $l_k$ is an $l$-rigid  isometric embedding of $\Gamma_{k}(V)$ in $\Gamma_{k}(V'/S)$.
As in the previous section, we denote by $\pi$ the mapping which transfers every subspace of $V'/S$ to the
corresponding subspace of $V'$.
Then $\pi l_k$ is an $l$-rigid isometric embedding of $\Gamma_{k}(V)$ in $\Gamma_{k'}(V')$ of type (A).
}\end{exmp}

\begin{exmp}\label{exmp6-3}{\rm
Suppose that
$$n\le k+k'\le n'$$
and $U\in {\mathcal G}_{k+k'}(V')$.
Let $v:V\to U^*$ be a semilinear embedding.
Then $v_k$ is  an $l$-rigid isometric embedding of $\Gamma_{k}(V)$ in $\Gamma_{k}(U^*)$.
By duality, it can be considered as an isometric embedding of $\Gamma_{k}(V)$ in $\Gamma_{k'}(U)$.
Lemma \ref{lemma6-1} guarantees that the latter embedding is $l$-rigid.
Since $U$ is contained in $V'$, we get an $l$-rigid isometric embedding of $\Gamma_{k}(V)$ in $\Gamma_{k'}(V')$ of type (B).
}\end{exmp}

It follows from Theorem \ref{theorem4-1} and the examples given above that
{\it every isometric embedding of $\Gamma_{k}(V)$ in $\Gamma_{k'}(V')$ is $l$-rigid
if $n=2k$.}
Now, we describe all $l$-rigid isometric embeddings of $\Gamma_{k}(V)$ in $\Gamma_{k'}(V')$  in the case when $k\ne n-k$.

\begin{theorem}\label{theorem6-1}
If $f$ is an  $l$-rigid isometric embedding of $\Gamma_{k}(V)$ in $\Gamma_{k'}(V')$
then one of the following possibilities is realized:
\begin{enumerate}
\item[{\rm (A)}] $k\le k'$, $n-k\le n'-k'$
and there is $S\in {\mathcal G}_{k'-k}(V')$
such that $f$ is induced by a semilinear embedding of $V$ in $V'/S$, see Example \ref{exmp6-2};
\item[{\rm (B)}] $n\le k+k'\le n'$ and there is $U\in {\mathcal G}_{k'+k}(V')$
such that $f$ is induced by a semilinear  embedding of $V$ in $U^{*}$, see Example \ref{exmp6-3}.
\end{enumerate}
\end{theorem}

The following example shows that there are isometric embeddings of $\Gamma_{k}(V)$ in $\Gamma_{k'}(V')$
which are not $l$-rigid.

\begin{exmp}\label{exmp6-4}{\rm
Suppose that $R=F$ and $R'=F'$, where $F$ and $F'$ are the fields from Remark \ref{rem-main1}, and $n>n'$.
Let $l:V\to V'$ be a semilinear $n'$-embedding.
If $n'\ge 2k$ then $l_k$ is an isometric embedding of  $\Gamma_{k}(V)$ in $\Gamma_{k}(V')$.
By Theorem \ref{theorem6-1}, this embedding is not $l$-rigid.
}\end{exmp}

\begin{proof}[Proof of Theorem \ref{theorem6-1}]
Let $f$ be an $l$-rigid isometric embedding of $\Gamma_{k}(V)$ in $\Gamma_{k'}(V')$.

{\it Case $k<n-k$}.

If $f$ is an embedding of type (A) then, by Theorem \ref{theorem4-1},
there exist $S\in {\mathcal G}_{k'-k}(V')$ and a semilinear $(2k)$-embedding $l:V\to V'/S$ such that
$f$ transfers every $X\in {\mathcal G}_{k}(V)$ to the subspace of $V'$ corresponding to $\langle l(X)\rangle$.
Since $f$ is $l$-rigid, $l_k$ is $l$-rigid
and Theorem \ref{theorem5-1} implies that $l$ is a semilinear embedding.

In the case when $f$ is an embedding of type (B), Theorem \ref{theorem4-1} implies
the existence of $U\in{\mathcal G}_{k'+k}(V')$ and a semilinear $(2k)$-embedding $v:V\to U^{*}$
such that $f$ transfers every $X\in {\mathcal G}_{k}(V)$ to the annihilator of $\langle v(X)\rangle$ in $U$.
Since $f$ is an $l$-rigid embedding of $\Gamma_{k}(V)$ in $\Gamma_{k'}(U)$,
Lemma \ref{lemma6-1} implies that $v_k$ is $l$-rigid.
By Theorem \ref{theorem5-1}, $v$ is a semilinear embedding.

{\it Case $k>n-k$}.

Let $f$ be an embedding of type (A).
Consider the embedding of $\Gamma_{n-k}(V^*)$ in $\Gamma_{k'}(V')$ transferring every $X\in {\mathcal G}_{n-k}(V^*)$ to $f(X^0)$.
By Lemma \ref{lemma6-1}, this embedding is $l$-rigid; moreover, it is an embedding of type (B).
Therefore, there exist $U\in {\mathcal G}_{k'+n-k}(V')$ and a semilinear embedding $v:V^{*}\to U^{*}$
such that $f$ transfers every $X\in {\mathcal G}_{k}(V)$ to the annihilator of $\langle v(X^{0})\rangle$ in $U$.
Since $v$ is a semilinear embedding, the image of $v$ spans an $n$-dimensional subspace of $U^{*}$.
Denote by $S$ the annihilator of this subspace in $U$. Then $S\in {\mathcal G}_{k'-k}(V')$
and the image of $f$ is contained in $[S\rangle_{k}$.

Consider the mapping $g$ which sends every subspace $X\subset V$ to the annihilator of $\langle v(X^{0})\rangle$ in $U$.
This is an injection of the set of all subspaces of $V$ to the set of all subspaces of $U$ containing $S$.
The image of every ${\mathcal G}_{i}(V)$ is contained in $[S\rangle_{k'-k+i}$
and the restriction of $g$ to ${\mathcal G}_{k}(V)$ coincides with $f$. The mapping $g$ is inclusions preserving:
for any subspaces $X,Y\subset V$
$$X\subset Y\; \Longrightarrow\; g(X)\subset g(Y).$$
This implies that the restriction of $g$ to ${\mathcal G}_{1}(V)$ transfers every line of $\Pi_{V}$ to a subset in a line
of the projective space $[S\rangle_{k'-k+1}$. By Theorem \ref{theorem2-1}, this restriction is induced
by a semilinear injection $l:V\to V'/S$. We need to show that for every $X\in {\mathcal G}_{k}(V)$
$$g(X)=\pi(\langle l(X)\rangle),$$
where $\pi$ transfers every subspace of $V'/S$ to the corresponding subspace of $V'$.

Since $v$ is a semilinear embedding, $g$ sends every base of $\Pi_{V}$ to an independent subset of
the projective space $[S\rangle_{k'-k+1}$. This implies that $l$ is a semilinear embedding.
Let $X\in {\mathcal G}_{k}(V)$. We take $P_{1},\dots,P_{k}\in {\mathcal G}_{1}(V)$ such that
$$X=P_{1}+\dots+P_{k}.$$
Then
$$\pi(\langle l(X)\rangle)=g(P_{1})+\dots+g(P_{k})\subset g(X)$$
and we get the required equality (since our subspaces both are $k'$-dimensional).

Now, let $f$ be an embedding of type (B). 
We consider $f$ as an embedding of $\Gamma_{k}(V)$ in $\Gamma_{n'-k'}(V'^{*})$. 
This is an embedding of type (A) and, as in the proof of Theorem \ref{theorem4-1},
we establish that the image of $f$ is contained in
$$\langle U]_{k'},\;\;U\in{\mathcal G}_{k'+k}(V').$$
So, $f$ is an $l$-rigid isometric embedding of $\Gamma_{k}(V)$ in $\Gamma_{k'}(U)$. By duality, it can be considered as
an isometric embedding of $\Gamma_{k}(V)$ in $\Gamma_{k}(U^*)$ of type (A).
Lemma \ref{lemma6-1} implies that the latter embedding is $l$-rigid.
Thus it is induced by a semilinear embedding of $V$ in $U^{*}$.
\end{proof}

\end{document}